\documentclass[12pt,oneside,reqno]{amsart}
\usepackage{amsthm,amscd,amssymb,verbatim,amsfonts,epsf,amsmath,mathrsfs,graphicx}
\usepackage{mathtools}
\usepackage{enumerate}
\usepackage[colorlinks=true,linkcolor=blue,citecolor=blue]{hyperref}
\def\line#1{\hbox to \hsize{#1\hfill}}
\theoremstyle{plain}
\newtheorem{conj}{Conjecture}
\newtheorem{prop}{Proposition}

\newtheorem{lemm}{Lemma}
\newtheorem{theo}{Theorem}

\newtheorem{coro}{Corollary}

\title[A para-K\"ahler structure in the space of oriented geodesics]%
{A para-K\"ahler structure in the space of oriented geodesics in a real space form}

\begin{document} 

\author{Nikos Georgiou}
\address{Nikos Georgiou\\
  Department of Mathematics\\
          Waterford Institute of Technology\\
          Waterford\\
          Co. Waterford\\
          Ireland.}
\email{ngeorgiou@wit.ie }

\begin{abstract}
In this article, we construct a new para-K\"ahler structure $({\mathcal G},{\mathcal J},\Omega)$ in the space of oriented geodesics ${\mathbb L}(M)$ in a non-flat, real space form $M$. We first show that the para-K\"ahler metric ${\mathcal G}$ is scalar flat and when $M$ is a 3-dimensional real space form, ${\mathcal G}$ is locally conformally flat. Furthermore, we prove that the space of oriented geodesics in hyperbolic $n$-space, equipped with the constructed metric ${\mathcal G}$, is minimally isometric embedded in the tangent bundle of the hyperbolic $n$-space. We then study the submanifold theory, and we show that ${\mathcal G}$-geodesics correspond to minimal ruled surfaces in the real space form. Lagrangian submanifolds (with respect to the canonical symplectic structure $\Omega$) play an important role in the geometry of the space of oriented geodesics as they are the Gauss map of hypersurfaces in the corresponding space form.  We demonstrate that the Gauss map of a non-flat hypersurface of constant Gauss curvature is a minimal Lagrangian submanifold. Finally, we show that a Hamiltonian minimal submanifold is locally the Gauss map of a hypersurface $\Sigma$ that is a critical point of the functional $\mathcal{F}(\Sigma)=\int_{\Sigma}\sqrt{|K|}\,dV$, where $K$ denotes the Gaussian curvature of $\Sigma$.\end{abstract}

\date{23\,rd November 2019}
\maketitle
\tableofcontents
\section{Introduction}
The geometry of the space ${\mathbb L}({\mathbb S}^{n+1}(c))$ of oriented geodesics in the real space form ${\mathbb S}^{n+1}(c)$, of constant sectional curvature $c$, has been of great interest for the last two decades. In the celebrated article \cite{gk1}, Guilfoyle and Klingenberg have constructed a K\"ahler structure in the space ${\mathbb L}({\mathbb S}^{3}(0))$ of oriented lines in the Euclidean 3-space ${\mathbb S}^{3}(0)={\mathbb R}^{3}$ and they showed that the K\"ahler metric is of neutral signature. Additionally, it is invariant under the group action of the Euclidean isometry group. A similar construction for the Hyperbolic 3-space ${\mathbb S}^{3}(-1)$ was established by Georgiou and Guilfoyle in \cite{gag}. Then, in \cite{An}, Anciaux used the fact that ${\mathbb L}({\mathbb S}^{n+1}(c))$ is identified with the Grassmannian of oriented two-planes of ${\mathbb R}^{n+2}$ to extend this geometric construction for all non-flat real space forms. In particular, he showed that ${\mathbb L}({\mathbb S}^{n+1}(c))$ admits a K\"ahler or a para-K\"ahler structure, where the metric (which will be denoted here by ${\mathcal G}_e$) is Einstein and invariant under the isometry group of  ${\mathbb S}^{n+1}(c)$. In the same work, Anciaux has proved that ${\mathbb L}({\mathbb S}^{3}(c))$ admits another K\"ahler or para-K\"ahler structure, such that its metric ${\mathcal G}_0$, is of neutral signature, locally conformally flat and is invariant under the isometry group of ${\mathbb S}^{3}(c)$. The geometry derived by the metric $\mathcal{G}_0$ has been studied by several authors (see for example, \cite{An, agk, gag, gk1, Salvai}).

The invariance of the constructed (para-) K\"ahler metric in ${\mathbb L}({\mathbb S}^{n+1}(c))$ under the action of the isometric group of ${\mathbb S}^{n+1}(c)$  allows one to study geometric problems in the base manifold ${\mathbb S}^{n+1}(c)$ by studying its space of oriented geodesics. For example, the set of all oriented geodesics orthogonal (called as the Gauss map) to a hypersurface in ${\mathbb S}^{n+1}(c)$ corresponds to a Lagrangian submanifold in ${\mathbb L}({\mathbb S}^{n+1}(c))$, with respect to the canonical symplectic structure $\Omega$ (see \cite{An}). In particular, ${\mathcal G}_0$ -flat Lagrangian surfaces in ${\mathbb L}({\mathbb S}^{3}(c))$ are the Gauss map of Weingarten surfaces in ${\mathbb S}^{3}(c))$, i.e. its principal curvatures are functionally related. If ${\mathbb L}({\mathbb S}^{n+1}(c))$ is equipped with the (para-) K\"ahler Einstein structure, then every Lagrangian submanifold admits a Lagrangian angle, that is, its corresponding Maslov 1-form is closed. This is generally true for any (para-) K\"ahler Einstein manifold. It is natural then to ask whether the converse is true, i.e. considering a (para-) K\"ahler manifold  $(M^{2n},g,J,\omega)$ such that every Lagrangian submanifold has closed Maslov 1-form, can we conclude that $g$ is Einstein?

 In this article, we show that the converse is not true. In particular, we construct a para-K\"ahler (non-Einstein) structure $(\mathcal{G},\mathcal{J},\Omega)$, where $\mathcal{G},\mathcal{J}$ and $\Omega$ are respectively the metric, the paracomplex structure and the canonical symplectic structure in ${\mathbb L}({\mathbb S}^{n+1}(c))$ such that all Lagrangian submanifolds have closed Maslov 1-form.  About the metric $\mathcal{G}$ we first show the following:
 
 \vspace{0.1in}
 
 ${\bf Theorem\; 1.}$ \emph{The metric $\mathcal{G}$ is scalar flat and non-Einstein. Furthermore, it is locally conformally flat if and only if $n=2$.}

\vspace{0.1in}

It is proved in \cite{gag2} that  $({\mathbb L}({\mathbb S}^{3}(0)),{\mathcal G}_0)$ is minimally isometrically embedded in the tangent bundle  
$(T{\mathbb S}^{3}(0), \overline{\mathcal{G}})$, where $\overline{\mathcal{G}}$ is a neutral, scalar flat and locally conformally flat metric. The following theorem provides a similar result for the hyperbolic case:
 
 \vspace{0.1in}
 
 ${\bf Theorem\; 2.}$ \emph{The isometric embedding $f:({\mathbb L}({\mathbb H}^{n+1}),{\mathcal G})\longrightarrow (T{\mathbb H}^{n+1},\overline{\mathcal{G}}):x\wedge y\mapsto (x,-y)$ is minimal.}
 
 \vspace{0.1in}
 
 The reason why Theorem 2 works for any dimension, while in the Euclidean case it only works for $n=2$, is because ${\mathbb L}({\mathbb H}^{n+1})$ admits invariant (para-) K\"ahler structures for any $n$. The space ${\mathbb L}({\mathbb R}^{3})$ of oriented lines in ${\mathbb R}^{n+1}$ admits an invariant (para-) K\"ahler structure only when $n=3$ and $7$ (see \cite{Salvai}). On the other hand, there is no similar result for the spherical cases since the spheres are not Hadamard and therefore the embedding $f$ is not well-defined.
 
\vspace{0.1in}

A curve in ${\mathbb L}({\mathbb S}^{3}(c))$ corresponds to a 1-parameter family of oriented geodesics, i.e. it corresponds to a ruled surface in the real space form ${\mathbb S}^{3}(c)$. The following theorem characterises the ${\mathcal G}$-geodesics:

 \vspace{0.1in}
 
 ${\bf Theorem\; 3.}$ \emph{A curve $\gamma$ in $({\mathbb L}({\mathbb S}^{3}(c)),{\mathcal G})$ is a geodesic if and only if the corresponding ruled surface in ${\mathbb S}^{3}(c)$ is minimal.}
 
 \vspace{0.1in}
 
 A \emph{Lagrangian submanifold} of a symplectic manifold $(M^{2n},\omega)$ is a submanifold $\Sigma^n$ of half dimension whose symplectic structure $\omega$ induced on  $\Sigma^n$ vanishes. It is already known that the Gauss map of any hypersurface in ${\mathbb S}^{3}(c)$ is a Lagrangian submanifold of $({\mathbb L}({\mathbb S}^{3}(c)),\Omega)$ where $\Omega:={\mathcal G}({\mathcal J}.,.)$ is the canonical symplectic structure. The following theorem describes all Lagrangian submanifolds that are minimal:
 
 \vspace{0.1in}
 
 ${\bf Theorem\; 4.}$ \emph{Every Lagrangian submanifold in $({\mathbb L}({\mathbb S}^{n+1}(c)),\mathcal{G},\Omega)$ has a Lagrangian angle. If $\Sigma$ is a non-flat hypersurface of ${\mathbb S}^{n+1}(c)$ then it is of constant Gaussian curvature if and only if the Gauss map of $\Sigma$ is a minimal Lagrangian submanifold of $({\mathbb L}({\mathbb S}^{n+1}(c)),\mathcal{G},\Omega)$.}
 
 \vspace{0.1in}
 
A non-flat point here we mean that the product of all principal curvatures is non-zero, i.e. the Gaussian curvature is non-zero.
 Theorem 4, shows that for $n=2$, the metrics ${\mathcal G}$ and $\mathcal{G}_0$ are not isometric since they have different minimal surfaces. In particular, the only $\mathcal{G}_0$-minimal Lagrangian surfaces in ${\mathbb L}({\mathbb S}^{3}(c))$ is the Gauss map of an equidistant tube along a geodesic, which is flat (see \cite{An,geo}).
 
 Additionally, the fact that all Lagrangian submanifolds admit a Lagrangian angle, implies that the converse of the question stated before, does not hold since ${\mathcal G}$ is not Einstein.

 Let $(M, \omega)$ be a symplectic manifold and $f:\Sigma\rightarrow M$ be a Lagrangian submanifold. A vector field $X$ along $\Sigma$ is said to be {\it Hamiltonian} if  the one form $f^{\ast}(V \rfloor\omega)$ is exact. A smooth variation $(f_t)$ of $\Sigma$ into $M$ is called a {\it Hamiltonian deformation} if $\frac{df}{dt}|_{t=0}$ is a Hamiltonian vector field.   
 
If it is also given a (para-) K\"ahler structure $(J, g, \omega)$ on $M$, then a normal variation $(f_t)$ of the Lagrangian submanifold $\Sigma$ is Hamiltonian if $\frac{df}{dt}|_{t=0} =J\nabla u$, where $J$ is the (para-) complex structure and $\nabla u$ is the gradient of the smooth function $u$ defined on $\Sigma$. A \emph{Hamiltonian minimal submanifold} is a Lagrangian submanifold that is a critical point of the volume functional with respect to Hamiltonian variations. The first variation formula of the volume functional implies that a Hamiltonian minimal submanifold is characterised by the equation $div JH=0$, where $H$ denotes the mean curvature vector of $\Sigma$ and div is the divergence operator with respect to the induced metric (for more details, see \cite{Oh} and \cite{Oh1}). It is proved in \cite{glob} and \cite{Palmer} that smooth one-parameter deformations of a submanifold in ${\mathbb S}^{n+1}(c))$ induce Hamiltonian deformations of the corresponding Gauss map in $({\mathbb L}({\mathbb S}^{n+1}(c)),\Omega)$. We then have the following:

 \vspace{0.1in}
 
  ${\bf Theorem\; 5.}${\it Let $\phi:\Sigma\rightarrow {\mathbb S}^{n+1}(c)$ be a non-flat hypersurface in ${\mathbb S}^{n+1}(c)$.  Then the Gauss map $\Phi:\Sigma\rightarrow {\mathbb L}({\mathbb S}^{n+1}(c))$ is a Hamiltonian minimal submanifold of $({\mathbb L}({\mathbb S}^{n+1}(c)),\Omega)$ if and only if $\phi$ is a critical point of the functional
 \[
 \mathcal{F}(\phi)=\int_{\Sigma}\sqrt{|K|}\,dV,
 \]
 where $K$ denotes the Gaussian curvature of $\phi$.}


\section{A canonical para K\"ahler structure}

Let ${\mathbb S}^{n+1}(c)$ be the real space form of constant sectional curvature $c\in\{-1,1\}$. Then, let ${\mathbb H}^{n+1}={\mathbb S}^{n+1}(-1)$ be the hyperbolic $(n+1)$-dimensional space defined by:
\[
{\mathbb H}^{n+1}=\{x\in {\mathbb R}^{n+2}\,|\, \left<x,x\right>_{-1}=1,\, x_0>0\},
\]
where $\left<x,x\right>_{-1}:=x_0^2-x_1^2-\ldots-x_{n+1}^2$ and ${\mathbb S}^{n+1}={\mathbb S}^{n+1}(1)$ be the $(n+1)$-dimensional sphere defined by:
\[
{\mathbb S}^{n+1}=\{x\in {\mathbb R}^{n+2}\,|\, \left<x,x\right>_{1}=1\},
\]
where $\left<x,x\right>_{1}:=x_0^2+x_1^2+\ldots+x_{n+1}^2$. Let $g$ be the flat metric $\left<.,.\right>_c$ induced in ${\mathbb S}^{n+1}(c)$ via the inclusion map.

The space of oriented geodesics in ${\mathbb S}^{n+1}(c)$ is identified with the Grassmannian of oriented two-planes of ${\mathbb R}^{n+2}$, i.e.,
\[
{\mathbb L}({\mathbb S}^{n+1}(c))=\{x\wedge y\in \Lambda^2({\mathbb R}^{n+2})\,|\, y\in T_x{\mathbb S}^{n+1}(c),\, \left<y,y\right>_c=c\}.
\] 
Every tangent vector in $T_{x\wedge y}{\mathbb L}({\mathbb S}^{n+1}(c))$ can be written as:
\[
x\wedge X+y\wedge Y,
\]
where $X,Y\in (x\wedge y)^{\bot}$ are in ${\mathbb R}^{n+2}$. It is known that ${\mathbb L}({\mathbb S}^{n+1}(c))$ is equipped with the Einstein metric $\mathcal{G}_e=\iota^{\ast}\left<\left<.,.\right>\right>_c$ where,
\[
\iota:{\mathbb L}({\mathbb H}^n)\xhookrightarrow{} \Lambda^2({\mathbb R}^{n+1}):x\wedge y\mapsto x\wedge y,
\]
and $\left<\left<.,.\right>\right>_c$ is the flat metric in $\Lambda^2({\mathbb R}^{n+1})$:
\[
\left<\left<x_1\wedge y_1,x_2\wedge y_2\right>\right>_c=\left<x_1,x_2\right>_c\left< y_1, y_2\right>_c-\left<x_1,y_2\right>_c\left<x_2,y_1\right>_c.
\]

\vspace{0.1in}

For the hyperbolic case ($c=-1$), fixing a point $p\in {\mathbb R}^{n+2}$, every oriented geodesic $\gamma=\gamma(t)$, with $t$ being its arc-length, can be identified with the pair $(\gamma(t_0),\gamma'(t_0))$, where $\gamma(t_0)$ is the closest point of $\gamma$ to $p$ and, $\gamma'(t_0)$ is its velocity. When $p$ is the origin, it is not hard to see that $\left<\gamma(t_0),\gamma'(t_0)\right>_1=0$.

\vspace{0.1in}

In this article, when we write the oriented geodesic $\gamma$ as the oriented plane $x\wedge y$ we mean that $\left< x, y\right>_1=0$.

\begin{prop}
The following embedding is well defined:
\begin{equation}\label{e:embedding}
f:{\mathbb L}({\mathbb H}^{n+1})\longrightarrow T{\mathbb H}^{n+1}:x\wedge y\mapsto (x,-y).
\end{equation}
\end{prop}
\begin{proof}
Indeed, let $z\wedge w\in {\mathbb L}({\mathbb H}^{n+1})$ be such that $z\wedge w=x\wedge y$, where $\left< x, y\right>_1=\left< z, w\right>_1=0$. Then 
\begin{equation}\label{e:equalityzandw}
z=x\cosh t+y\sinh t,\qquad w=x\sinh t+y\cosh t,
\end{equation}
for some real $t$. Note that $\left<x,y\right>_1=0$ and thus we have that $y_0=0$. The fact that $\left<y,y\right>_{-1}=-1$ implies $\left<y,y\right>_1=1$. 

From $\left<z,w\right>_1=0$, we then have,
\[
(|x|_1^2+|y|_1^2)\sinh t\cosh t+\left<x,y\right>_1(\cosh^2t+\sinh^2t)=0,
\]
which yields,
\[
(|x|_1^2+1)\sinh t\cosh t=0, 
\]
Thus, $t=0$ and substituting this in (\ref{e:equalityzandw}), we finally get $(x,-y)=(z,-w)$, which means that $f(x\wedge y)=f(z\wedge w)$.
\end{proof}

\vspace{0.1in}

We now use the embedding $f$ to define a new geometric structure on ${\mathbb L}({\mathbb H}^{n+1})$. To do this, consider the neutral metric $\overline{\mathcal{G}}$ on $T{\mathbb H}^{n+1}$:
\[
\overline{\mathcal{G}}(\bar X,\bar Y)=g(\Pi\bar X,K\bar Y)+g(K\bar X,\Pi\bar Y),
\]
where $\bar X\simeq (\Pi\bar X,K\bar X)$, $\bar Y\simeq (\Pi\bar Y,K\bar Y)$ in $TT{\mathbb H}^{n+1}=T{\mathbb H}^{n+1}\oplus T{\mathbb H}^{n+1}$, and $g$ is the metric $\left<.,.\right>_{-1}$ induced by the inclusion map $i:{\mathbb H}^{n+1}\hookrightarrow {\mathbb R}^{n+2}$. For more details about this metric, see \cite{gag2} and \cite{gudkappos}.

Let $\mathcal{G}$ be the metric $\overline{\mathcal{G}}$ induced by $f$ on ${\mathbb L}({\mathbb H}^{n+1})$, i.e. $\mathcal{G}=f^{\ast}\overline{\mathcal{G}}$. It can be shown that 
\begin{equation}\label{e:definofmetric}
\mathcal{G}(x\wedge X_1+y\wedge Y_1,x\wedge X_2+y\wedge Y_2)=g(X_1,Y_2)+g(X_2,Y_1).
\end{equation}

\begin{prop}
The metrics $\mathcal{G}$ and $\mathcal{G}_e$ share the same Levi-Civita connection.
\end{prop}

\begin{proof}
Let $x,y,e_1,\ldots e_{n}$ be an orthonormal frame of ${\mathbb R}^{n+2}$, and define the vector $E_1,\ldots E_{2n}$ in $T_{x\wedge y}{\mathbb L}({\mathbb H}^{n+1})$ by:
\[
E_i=x\wedge e_i,\qquad E_{n+i}=y\wedge e_i,
\]
where $i=1,\ldots, n$. If $\mathcal{D}$ is the Levi-Civita connection of $\mathcal{G}_e$, one can show that $\mathcal{D}_{E_i}E_j=0$.

An almost complex structure $\mathcal{J}_e$ in ${\mathbb L}({\mathbb H}^{n+1})$ can be defined by
\[
\mathcal{J}_e(x\wedge X+y\wedge Y)=-y\wedge X+x\wedge Y.
\]
Then $\mathcal{D}_{E_i}\mathcal{J}_e  = \mathcal{J}_e\mathcal{D}_{E_i}$, which shows that $\mathcal{J}_e$ is $\mathcal{D}$-parallel and therefore integrable. We also have that $\mathcal{J}_e$ is symmetric with respect to $G_0$, i.e.
\[
\mathcal{G}_e(\mathcal{J}_e\bar X,\bar Y)=\mathcal{G}_e(\bar X,\mathcal{J}_e\bar Y),
\]
for any $\bar X,\bar Y\in T_{x\wedge y}{\mathbb L}({\mathbb H}^{n+1})$. Namely,
\begin{eqnarray}
\mathcal{G}_e(\mathcal{J}_e(x\wedge X_1+y\wedge Y_1),x\wedge X_2+y\wedge Y_2)&=&\mathcal{G}_e(-y\wedge X_1+x\wedge Y_1,x\wedge X_2+y\wedge Y_2)\nonumber\\
&=& g(X_1,Y_2)+g(X_2,Y_1),\label{e:themetric}
\end{eqnarray}
which implies 
\[
\mathcal{G}_e(\mathcal{J}_e(x\wedge X_1+y\wedge Y_1),x\wedge X_2+y\wedge Y_2)=\mathcal{G}_e(x\wedge X_1+y\wedge Y_1,\mathcal{J}_e(x\wedge X_2+y\wedge Y_2)).
\]
Consider the following Lemma:
\begin{lemm}\cite{An}
Let $(N,G)$ be a pseudo-Riemannian manifold with Levi-Civita connection $D$ and $T$ a symmetric, $D$-parallel $(1,1)$ tensor.Then the Levi-Civita connection of the pseudo-Riemannian metric $G'=G(.,T.)$ is $D$.
\end{lemm}
From (\ref{e:themetric}), we have 
\[
\mathcal{G}=\mathcal{G}_e(.,\mathcal{J}_e.),
\]
and the proposition then follows.
\end{proof}

\vspace{0.1in}

Considering the $(n+1)$- dimensional real space form ${\mathbb S}^{n+1}(c)$ and defining the almost (para-)complex structure $\mathcal{J}_e$ by
\[
\mathcal{J}_e(x\wedge X+y\wedge Y)=cy\wedge X+x\wedge Y,
\] 
we define the metric $\mathcal{G}$ by
\[
\mathcal{G}=\mathcal{G}_e(.,\mathcal{J}_e.),
\]
which is given by (\ref{e:definofmetric}). It is easily shown that, $\mathcal{G}$ is $\mathcal{D}$ and $\mathcal{J}_e$ symmetric and therefore $\mathcal{G}$ and $\mathcal{G}_e$ share the same Levi-Civita connection.   

The following theorem explores the curvature of $\mathcal{G}$:

\begin{theo}
The metric $\mathcal{G}$ is scalar flat and non-Einstein. Furthermore, it is locally conformally flat if and only if $n=2$.
\end{theo}
\begin{proof}
Consider the frame $E_i$ used previously, where again $i=1,\ldots, n$, then $\mathcal{J}_eE_i=cE_{n+i}$ and $\mathcal{J}_eE_{n+i}=E_{i}$. Let $\mathcal{R}$ and $\mathcal{R}ic$ be the Riemann curvature and Ricci tensor respectively of $\mathcal{G}$. Since the metrics $\mathcal{G}$ and $\mathcal{G}_e$ have the same Levi-Civita connection then $\mathcal{R}=R$, where $R$ is the Riemann curvature tensor of $\mathcal{G}_e$. Then,
\begin{eqnarray}
\mathcal{G}(\mathcal{R}(.,.).,.)&=&\mathcal{G}_e(\mathcal{R}(.,.).,\mathcal{J}_e.)\nonumber \\
&=&\mathcal{G}_e(R(.,.).,\mathcal{J}_e.)\nonumber 
\end{eqnarray}
For $i,j=1,\ldots,n$ we have,
\[
\mathcal{G}_{i,n+j}=c\delta_{ij},\quad \mathcal{G}_{ij}=\mathcal{G}_{n+i,n+j}=0,
\]
and therefore the inverse matrix has coefficients
\[
\mathcal{G}^{i,n+j}=c\delta_{ij},\quad \mathcal{G}^{ij}=\mathcal{G}^{n+i,n+j}=0.
\]
Using the fact that $\mathcal{G}_e^{ij}=c\mathcal{G}_e^{n+i,n+j}=\delta_{ij}$ and $\mathcal{G}_e^{i,n+j}=0$, we then have
\begin{eqnarray}
\mathcal{R}ic(X,Y)&=&\sum_{i=1}^{n}\mathcal{G}^{i,n+i}\left(\mathcal{G}(\mathcal{R}(X,E_i)Y,E_{n+i})+\mathcal{G}(\mathcal{R}(X,E_{n+i})Y,E_{i})\right)\nonumber \\
&=&\sum_{i=1}^{n}\left(\mathcal{G}_e(R(X,E_i)Y,\mathcal{J}_eE_{n+i})+\mathcal{G}_e(R(X,E_{n+i})Y,\mathcal{J}_eE_{i})\right)\nonumber \\
&=&\sum_{i=1}^{n}\Big(\left<\left<R(X,E_i)Y,E_{i}\right>\right>_c+c\left<\left<R(X,E_{n+i})Y,E_{n+i}\right>\right>_c\Big)\nonumber \\
&=&\sum_{i=1}^{n}\Big(\mathcal{G}_e^{ii}\left<\left<R(X,E_i)Y,E_{i}\right>\right>_c+\mathcal{G}_e^{n+i,n+i}\left<\left<R(X,E_{n+i})Y,E_{n+i}\right>\right>_c\Big)\nonumber \\
&=& Ric(X,Y),\nonumber
\end{eqnarray}
where $Ric$ is the Ricci tensor of $\mathcal{G}_e$.

Now, $\mathcal{G}_e$ is an Einstein metric with scalar curvature $S=2cn^2$ (for more details, see \cite{An}). Then
\[
Ric=\frac{S}{2n}\mathcal{G}_e=cn\,\mathcal{G}_e.
\]
That means,
\[
\mathcal{R}ic(X,Y)=cn\,\mathcal{G}_e(X,Y)
\]
which implies,
\[
\mathcal{R}ic(X,Y)=cn\,\left<\left<X,Y\right>\right>_c,
\]
and thus, $\mathcal{G}$ is non-Einstein. 

If $\mathcal{S}$ denotes the scalar curvature of $\mathcal{G}$ then,
\begin{eqnarray}
\mathcal{S}&=&\sum_{a,b=1}^{2n}\mathcal{G}^{ab}\mathcal{R}ic(E_a,E_b)\nonumber \\
&=&2\sum_{i=1}^{n}\mathcal{G}^{i,n+i}\mathcal{R}ic(E_{i},E_{n+i})\nonumber\\
&=&2c^2n\sum_{i}^{n}\left<\left<E_{i},E_{n+i}\right>\right>_c\nonumber\\
&=&0.\nonumber
\end{eqnarray}

We now proceed with the proof of the second part of the theorem. Since $\mathcal{G}$ is scalar flat, the Weyl tensor $\mathcal{W}$ is given by
\begin{eqnarray}
\mathcal{W}(X,Y,Z,W)&=&\mathcal{G}(\mathcal{R}(X,Y)Z,W)-\frac{1}{2(n-1)}\mathcal{R}ic\circ \mathcal{G}(X,Y,Z,W)\nonumber\\
&=&\mathcal{G}_e(R(X,Y)Z,\mathcal{J}_eW)-\frac{1}{2(n-1)}Ric\circ \mathcal{G}(X,Y,Z,W)\nonumber\\
&=&\mathcal{G}_e(R(X,Y)Z,\mathcal{J}_eW)-\frac{cn}{2n-2}\,\mathcal{G}_e\circ \mathcal{G}(X,Y,Z,W),\nonumber
\end{eqnarray}
where $\circ$ denotes the Kulkarni-Nomizu product in symmetric 2-tensors.

Now, 
\begin{eqnarray}
\mathcal{W}(E_1,E_2,E_2,E_{n+1})&=&\mathcal{G}_e(R(E_1,E_2)E_2,\mathcal{J}_eE_{n+1})-\frac{cn}{2n-2}\mathcal{G}_e\circ \mathcal{G}(E_1,E_2,E_2,E_{n+1})\nonumber\\
&=&\mathcal{G}_e(R(E_1,E_2)E_2,E_{1})-\frac{cn}{2n-2}\mathcal{G}_e(E_2,E_2)\mathcal{G}(E_1,E_{n+1})\nonumber\\
&=&1-\frac{c^2n}{2n-2}=1-\frac{n}{2n-2},\nonumber
\end{eqnarray}
which is zero if and only if $n=2$. Similarly, one can prove the same for the other coefficients of the Weyl tensor.
\end{proof}



\vspace{0.1in}

Finally, for the hyperbolic case ($c=-1$), we show the following:
\begin{theo}
The isometric embedding $f:({\mathbb L}({\mathbb H}^{n+1}),{\mathcal G})\longrightarrow (T{\mathbb H}^{n+1},\overline{{\mathcal G}}):x\wedge y\mapsto (x,-y)$ is minimal.
\end{theo}
\begin{proof}
The derivative of $f$ is given by:
\[
df(x\wedge X+y\wedge Y)=(-Y,-X).
\]
Note that $X=D_{Y}y$ and if $\overline{\mathcal{D}}$ denotes the Levi-Civita connection of $\overline{\mathcal{G}}$, we have
\begin{eqnarray}
\overline{\mathcal{D}}_{df(x\wedge X_1+y\wedge Y_1)}df(x\wedge X_1+y\wedge Y_1)&=&\overline{\mathcal{D}}_{(-Y_1,-X_1)}(-Y_2,-X_2)\nonumber\\
&=&(D_{Y_1}Y_2\,,\, R(y,Y_1)Y_2+D_{Y_1}X_2)\nonumber\\
&=&(D_{Y_1}Y_2\,,\, g(Y_1,Y_2)y+D_{Y_1}X_2)\nonumber\\
&=&(D_{Y_1}Y_2\,,\, -g(X_1,X_2)y+D_{Y_1}X_2)\nonumber\\
&&\qquad\qquad+(0,(g(X_1,X_2)+g(Y_1,Y_2))y),\nonumber
\end{eqnarray}
which implies that the second fundamental form $h_f$ is given by
\[
h_f(x\wedge X_1+y\wedge Y_1,x\wedge X_2+y\wedge Y_2)=(0,(g(X_1,X_2)+g(Y_1,Y_2))y).
\]
Recalling the basis $(E_1,\ldots E_{2n})$ of $T_{x\wedge y}{\mathbb L}({\mathbb S}^{n+1}(c))$, the mean curvature $H_f$ of $f$ is 
\[
H_f={\mathcal G}^{mn}h_f(E_m,E_n),
\]
so that
\begin{eqnarray}
H_f&=&{\mathcal G}^{i,n+i}h_f(E_i,E_{n+i})\nonumber \\
&=&h_f(x\wedge e_i,y\wedge e_i)\nonumber \\
&=&(0,(g(e_i,0)+g(0,e_i))y),\nonumber
\end{eqnarray}
which shows that $f$ is minimal.
\end{proof}

Considering the following almost para-complex structure ${\mathcal J}$ in ${\mathbb L}({\mathbb S}^{n+1}(c))$:
\[
{\mathcal J}(x\wedge X+y\wedge Y)=x\wedge X-y\wedge Y,
\]
we then have:
\begin{enumerate}
\item ${\mathcal J}$ is compatible with $\mathcal{G}$. Namely,
\begin{eqnarray}
\mathcal{G}({\mathcal J}(x\wedge X_1+y\wedge Y_1),{\mathcal J}(x\wedge X_2+y\wedge Y_2))&=&\mathcal{G}(x\wedge X_1-y\wedge Y_1,x\wedge X_2-y\wedge Y_2)\nonumber \\
&=&-g(X_1,Y_2)-g(X_2,Y_1)\nonumber \\
&=&-\mathcal{G}(x\wedge X_1+y\wedge Y_1,x\wedge X_2+y\wedge Y_2).\nonumber
\end{eqnarray}
\item ${\mathcal J}$ is integrable, i.e., $\mathcal{D}{\mathcal J}={\mathcal J}\mathcal{D}$. In fact,
\[
{\mathcal J}E_i=E_i,\qquad {\mathcal J}E_{n+i}=-E_{n+i}
\]
and the claim follows from $\mathcal{D}_{E_i}E_{j}=0$.


\vspace{0.1in}

\end{enumerate}
Define the symplectic 2-form $\Omega$ in ${\mathbb L}({\mathbb S}^{n+1})$ by
\[
\Omega=\mathcal{G}({\mathcal J}.,.).
\]
In particular, 
\[
\Omega(x\wedge X_1+y\wedge Y_1,x\wedge X_2+y\wedge Y_2)=g(X_1,Y_2)-g(X_2,Y_1).
\]
 
Then the quadruple $({\mathbb L}({\mathbb S}^{n+1}(c)),\mathcal{G}, \mathcal{J},\Omega)$ form a para-K\"ahler structure, so that the symplectic structure is the same with symplectic structure defined by the (para-) K\"ahler structure $({\mathbb L}({\mathbb S}^{n+1}(c)),\mathcal{G}_e, \mathcal{J}_e)$, since
\[
\Omega=\mathcal{G}_e({\mathcal J}_e.,.).
\]
The latter (para-) K\"ahler structure has been widely studied in \cite{An}, \cite{gag} and \cite{gk1}.

Every isometry $\phi:{\mathbb S}^{n+1}(c)\rightarrow {\mathbb S}^{n+1}(c)$, can be extended to a linear orthogonal transformation $\bar\phi$ in ${\mathbb R}^{n+2}$ restricted into ${\mathbb S}^{n+1}(c)$. This induces a mapping $F$ in the space of oriented geodesics defined by
\[
F(x\wedge y)=\phi(x)\wedge \bar\phi(y).
\]
The derivative of $F$ is
\[
dF(x\wedge X+y\wedge Y)=\phi(x)\wedge d \bar\phi(X)+\bar\phi(y)\wedge d \phi(Y).
\]
Using now the fact that $X,Y\in (x\wedge y)^{\bot}$ (see \cite{An}), we have that $X\in T_{x}{\mathbb S}^{n+1}(c)$ and thus,
\[
dF(x\wedge X+y\wedge Y)=\phi(x)\wedge d\phi(X)+\bar\phi(y)\wedge d \phi(Y).
\]
We now have
\[
\mathcal{G}(dF(x\wedge X_1+y\wedge Y_1),dF(x\wedge X_2+y\wedge Y_2))
\]
\[
=\mathcal{G}(\phi(x)\wedge d \phi(X_1)+\bar \phi(y)\wedge d\phi(Y_1),\phi(x)\wedge d \phi(X_2)+\bar \phi(y)\wedge d\phi(Y_2))
\]
\[
=g(d \phi(X_1),d \phi(Y_2))+g(d \phi(X_2),d\phi(Y_1))
\]
\[
=g(X_1,Y_2)+g(X_2,Y_1)
\]
\[
=\mathcal{G}(x\wedge X_1+y\wedge Y_1,x\wedge X_2+y\wedge Y_2),
\]
which shows the following:
\begin{prop}
The metric $\mathcal{G}$ is invariant under the group action of the isometry group of $({\mathbb S}^{n+1}(c),g)$ in the space of oriented geodesics ${\mathbb L}({\mathbb S}^{n+1}(c))$.
\end{prop}


\vspace{0.2in}

\section{Geodesics}

We now study geodesics in $({\mathbb L}({\mathbb S}^{n+1}(c)),\mathcal{G})$. We start with the following proposition:

\begin{prop}
If the curve $\gamma(t)=x(t)\wedge y(t)$ is a $\mathcal{G}$-geodesic ${\mathbb L}({\mathbb S}^{n+1}(c))$, then the vector field $y=y(t)$ is orthogonal to the curve $x=x(t)$ in ${\mathbb S}^{n+1}(c)$.
\end{prop}
\begin{proof} We prove the proposition for $c=1$, as the proof is similar for $c=-1$. 
Denote by $\overline\nabla$ the flat connection of $\Lambda^2{\mathbb R}^{n+2}$ and $D$ the Levi-Civita connection of $g$. Then 
\[
\overline\nabla_{\dot\gamma}\dot\gamma=D_{\dot x}\dot x\wedge y+x\wedge (D^2_{\dot x}y+\left<\dot x,y\right>_1\dot x)-x\wedge y+2\dot x\wedge D_{\dot x}y.
\]
If $\mathcal{D}$ is the Levi-Civita connection of $\mathcal{G}$, we then have:
\[
\mathcal{D}_{\dot\gamma}\dot\gamma=D_{\dot x}\dot x\wedge y+x\wedge (D^2_{\dot x}y+\left<\dot x,y\right>_1\dot x).
\]
Suppose $\gamma$ is a $\mathcal{G}$-geodesic. Then 
\[
D_{\dot x}\dot x=ay\quad\mbox{and}\quad D^2_{\dot x}y+\left<\dot x,y\right>_1\dot x=by,
\]
for some functions $a=a(t),\,b=b(t)$ along the curve $x=x(t)$. Assuming $t$ is the arc-length of the curve $x$, it follows that
\[
0=g(D_{\dot x}\dot x,\dot x)=ag(\dot x,y).
\]  
If $a\neq 0$ at in some open interval, then obviously we have that $g(\dot x,y)=0$. Assuming $a=0$ in an open interval, we have that $x$ is a geodesic at in that interval. Note that $\dot x,y$ are linearly independent, since otherwise it can be shown that $y=\pm \dot x$ and therefore the curve $\gamma(t)=\pm x\wedge\dot x$ is not regular.

 Let $x,\dot x,y,e_1,\ldots e_{n-1}$ be a frame of ${\mathbb R}^{n+2}$ such that $g(e_i,e_j)=\delta_{ij}$ and set $c_0=g(D_{\dot x}y,\dot x)$ with $c_k=g(D_{\dot x}y,e_k)$. 

Now, $g(D_{\dot x}y,D_{\dot x}y)=\displaystyle\sum_{k=0}c_k^2=-g(D^2_{\dot x}y,y)$, and therefore
$$\displaystyle\sum_{k=0}c_k^2=-b+g(\dot x,y)^2.$$
On the other hand
\[
by-g(\dot x,y)\dot x=D^2_{\dot x}y=D_{\dot x}(D_{\dot x}y)=D_{\dot x}(c_0\dot x+\sum_{k=1}c_ke_k)
\]
\[
\qquad\qquad\qquad\qquad\qquad\qquad=\dot{c}_0\dot x+\sum_k(\dot{c}_ke_k+c_k\dot{e}_k)=\dot{c_0}\dot x+\sum_k(\dot{c}_ke_k+c_k\dot{e}_k)
\]
\[
\qquad\qquad\qquad\qquad\qquad\qquad=\dot{c_0}\dot x-(\sum_{k=1}c_k^2)y+\Lambda(e_1,\ldots,e_{n-1}),
\]
where, $\Lambda\in\mbox{span}\{e_1,\ldots,e_{n-1}\}$. Then
\[
\dot{c}_0=-g(\dot x,y),\quad b=-\sum_{k=1}c_k^2,
\]
and $\Lambda=0$. In particular, for every $k=1,\ldots, n-1$, we have
\[
\dot{c}_k+\sum_{i\neq k}g(e_i,\dot{e}_k)c_i=0.
\]
Thus,
\[
\sum_{k=1}^{n-1}c_k\dot{c}_k=\sum_{k,i=1}^{n-1}c_ic_kg(e_i,\dot{e}_k)=0,
\]
which implies that $\displaystyle\sum_{k=1}^{n-1}c^2_k=\mbox{constant}$. This means, $b=-\displaystyle\sum_{k=1}^{n-1}c_k^2$ is constant and by definition we have
\[
b=g(D^2_{\dot x}y+\left<\dot x,y\right>\dot x,y)=-g(D_{\dot x}y,D_{\dot x}y)+\left<\dot x,y\right>_1^2.
\]
Using now the fact that $b$ is constant we have
\[
\dot b=4g(\dot x,y)g(\dot x,D_{\dot x}y)=2\frac{d}{dt}\Big(g(\dot x,y)^2\Big)=0.
\]
It follows $g(\dot x,y)$ is constant and therefore $g(\dot x,D_{\dot x}y)=0$. Therefore, $c_0=0$ since
$c_0=g(D_{\dot x}y,\dot x)$. But $0=\dot{c}_0=-g(\dot x,y)$ and the proposition follows.
\end{proof}

\vspace{0.1in}

Every curve $\gamma=\gamma(t)=x(t)\wedge y(t)$ in ${\mathbb L}({\mathbb S}^{n+1}(c))$, corresponds to a ruled surface in ${\mathbb S}^{n+1}(c)$ and such a surface, can be parametrised by 
\begin{equation}\label{e:parametr}
X(t,\theta)=x(t)\cos c(\theta)+y(t)\sin c(\theta),
\end{equation}
where,
\[
\cos c(\theta)=\begin{cases} \cos(\theta), & c=1\\ \cosh(\theta), & c=-1\end{cases}
\]
For $n=2$, we show the following:
\begin{theo}\label{t:geodesics}
A curve $\gamma$ in $({\mathbb L}({\mathbb S}^{3}(c)),{\mathcal G})$ is a geodesic if and only if the corresponding ruled surface in ${\mathbb S}^{3}(c)$ is minimal.
\end{theo}
\begin{proof}
We know that $\dot x,y$ are linearly independent and let, $\{x,\dot x,y, e_1\}$ be an orthonormal frame of ${\mathbb R}^4,\left<.,.\right>)$ along the curve $x=x(t)$. The corresponding ruled surface, parametrised by (\ref{e:parametr}), has normal vector fields $N$, where:
\[
N(t,\theta)=e_1-\frac{c_1\sin c\theta}{|X_t|^2}X_t,
\]
with $c_1=\left<D_{\dot x}y,e_1\right>_c$. Now
\[
N_{\theta}=-\frac{c_1\cos c\theta}{|X_t|^2}X_t,\quad N_{t}=\dot{e}_1-\frac{\dot{c}_1\sin c\theta}{|X_t|^2}X_t-\frac{c_1\sin c\theta}{|X_t|^2}X_{tt}+\frac{c_1\sin c\theta}{|X_t|^4}\left<X_{tt},X_t\right>X_t,,
\]
If $h$ is the second fundamental form of $X$, we then have
\[
h(X_t,X_t)=-\left<X_t,N_t\right>_c=\dot{c}_1\sin c\theta
\]
\[
h(X_t,X_{\theta})=-\left<X_t,N_{\theta}\right>_c=c_1\cos c\theta
\]
\[
h(X_{\theta},X_{\theta})=-\left<X_{\theta},N_{\theta}\right>_c=0.
\]
If $H$ is the mean curvature and $t_{ij}$ be the induced metric $X^{\ast}g$, we have that $t_{t\theta}=0$. Therefore
\[
H=\frac{1}{2}t^{ij}h(X_i,X_j)=\frac{1}{2}t^{tt}h(X_t,X_t)+\frac{1}{2}t^{\theta\theta}h(X_{\theta},X_{\theta}).
\]
Thus, using the previous proposition
\[
H=\frac{\dot{c}_1\sin c\theta}{2|X_t|^2}=0,
\]
since, $c_1$ is constant. 
\end{proof}

It would be interesting to know whether the Theorem \ref{t:geodesics} can be extended for any dimension, we therefore conjecture the following: 

\begin{conj}
A curve $\gamma$ in $({\mathbb L}({\mathbb S}^{n+1}(c)),{\mathcal G})$ is a geodesic if and only if the corresponding ruled surface in ${\mathbb S}^{n+1}(c)$ is minimal.
\end{conj}


\vspace{0.2in}

\section{Lagrangian submanifolds}

Let $\phi:\Sigma^n\rightarrow {\mathbb S}^{n+1}(c)$, be an immersed, orientable hypersurface and $N$ be the unit normal vector field along $\Sigma$. The Gauss map 
\[
\Phi:\Sigma\rightarrow {\mathbb L}({\mathbb S}^{n+1}(c)):x\mapsto\phi(x)\wedge N(x),
\]
defines a Lagrangian immersion in ${\mathbb L}({\mathbb S}^{n+1})(c)$ with respect to the symplectic structure $\Omega$. In the other hand, any Lagrangian immersion in ${\mathbb L}({\mathbb S}^{n+1}(c))$ is locally the Gauss map of a hypersurface in ${\mathbb S}^{n+1}(c))$ and hence, it is immersed by a mapping $\Phi$ (for more details, see Theorem 3 of \cite{An}).

Identifying any vector field $X$ in $\Sigma$ with the derivative $d\phi(X)$, we have
\[
\bar X=d\Phi(X)=X\wedge N+AX\wedge\phi,
\]
where $A$ denotes the shape operator of $\phi$. Let $\nabla$ and $\overline{\nabla}$ be the flat connections of ${\mathbb R}^{n+2}$ and $\Lambda^2{\mathbb R}^{n+2}$ respectively, then we get
\[
\overline{\nabla}_{\bar X}\bar Y=(\nabla_X Y)\wedge N+(\nabla_X AY)\wedge\phi. 
\]
Since the Levi-Civita connection $\nabla$ of $\mathcal{G}$ is the same with $\mathcal{G}_e$, the second fundamental form $\bar h$ of $\Phi$ is is described explicitly by the following tri-symmetric tensor:
\[
\bar h(\bar X,\bar Y,\bar Z)=\mathcal{G}(\mathcal{D}_{\bar X}\bar Y,\mathcal{J}\bar Z).
\]
Let $(e_1,\ldots,e_n)$ be the set of all principal directions of $(\Sigma,\phi^{\ast}g)$ with corresponding principal curvatures $k_i$. That is, $Ae_i=k_ie_i$, where $A$ is the shape operator of $\phi$. If we simply write the induced metric $\Phi^{\ast}\mathcal{G}$ as $\mathcal{G}$ then
\begin{equation}\label{e:inducmet}
\mathcal{G}(\bar e_i,\bar e_j)=2\delta_{ij}k_i.
\end{equation}
Away from flat points, i.e. $\Pi_{k=1}^nk_i\neq 0$, we have
\[
\bar h(e_i,e_j,e_j)=-e_i(k_j),
\]
and therefore, if ${\mathbb H}$ is the mean curvature of $\Phi$, we obtain
 
\[
\mathcal{G}(n{\mathbb H},\mathcal{J}d\Phi(e_i))=\sum_{i=1}^n\frac{\bar h(e_i,e_j,e_j)}{\mathcal{G}(e_j,e_j)}=-\sum_{i=1}^n\frac{e_i(k_j)}{2k_j}=e_i\log|k_1\cdot \ldots \cdot k_n|^{-1/2}.
\]
Hence, 
\[
{\mathbb H}=\frac{1}{n}\mathcal{J}\mathcal{D}\log|k_1\cdot \ldots \cdot k_n|^{-1/2},
\]
and finally we have that
\begin{equation}\label{e:meancurv}
{\mathbb H}=\frac{1}{n}\mathcal{J}\mathcal{D}\log|K|^{-1/2},
\end{equation}
where $K=\Pi_{i=1}k_i$, is the Gaussian curvature of the hypersurface $\phi$.

We thus obtain the following: 

\begin{theo}\label{t:minimal}
Every Lagrangian submanifold in $({\mathbb L}({\mathbb S}^{n+1}(c)),\mathcal{G},\Omega)$ has a Lagrangian angle. If $\Sigma$ is a non-flat hypersurface of ${\mathbb S}^{n+1}(c)$ then it is of constant Gaussian curvature if and only if the Gauss map of $\Sigma$ is a minimal Lagrangian submanifold of $({\mathbb L}({\mathbb S}^{n+1}(c)),\mathcal{G},\Omega)$.
\end{theo}

A Lagrangian submanifold $\Sigma$ is said to be {\it Hamiltonian minimal} if 
\[
\frac{d}{dt}\,\mbox{vol}f_t(\Sigma)|_{t=0}=0,
\]
for all Hamiltonian deformations $\{f_t\}$ of $\Sigma$. Using the first variation formula, $\Sigma$ is Hamiltonian minimal if 
\[
\delta a_{H}=0,
\]
where $H$ and $a_H$ are, respectively, the mean curvature vector and the Maslov 1-form, i.e.  $a_H=\Omega(H,.)$ and $\delta$ is the Hodge-dual of $d$. If the ambient manifold admits a (para-) K\"ahler structure $(G,J,\Omega)$ then, Y-G  Oh showed in \cite{Oh1} that a Lagrangian submanifold is Hamiltonian minimal if and only if $div JH=0$, where div is the divergence operator with respect to the induced metric.

 \vspace{0.1in}

\begin{theo}\label{t:hminimal}
Let $\phi:\Sigma\rightarrow {\mathbb S}^{n+1}(c)$ be a non-flat hypersurface in ${\mathbb S}^{n+1}(c)$.  Then the Gauss map $\Phi:\Sigma\rightarrow {\mathbb L}({\mathbb S}^{n+1}(c))$ is a Hamiltonian minimal submanifold of $({\mathbb L}({\mathbb S}^{n+1}(c)),\Omega)$ if and only if $\phi$ is a critical point of the functional
 \[
 \mathcal{F}(\phi)=\int_{\Sigma}\sqrt{|K|}\,dV,
 \]
 where $K$ denotes the Gaussian curvature of $\phi$.
\end{theo}

\begin{proof}
Let $\Phi$ be the Gauss map of a smooth immersion of $\phi$ of the $n$-dimensional manifold $\Sigma$ in ${\mathbb S}^{n+1}(c)$ and let $(e_1,\ldots,e_n)$ be an orthonormal frame, with respect to the induced metric $\phi^{\ast}g$, such that
\[
Ae_i=k_i e_i, \qquad i=1,\ldots,n,
\]
where $A$ denotes the shape operator of $\phi$.

 Let $(\phi_t)_{t\in (-t_0,t_0)}$ be a smooth variation of $\phi$ and $(\Phi_t)$ be the corresponded variation of the Gauss map $\Phi$. It has been proved in \cite{glob} and \cite{Palmer} that $(\Phi_t)$ is a Hamiltonian variation. Also the converse is true, i.e. Hamiltonian variations in ${\mathbb L}({\mathbb S}^{n+1}(c))$ are the Gauss maps of smooth variations in ${\mathbb S}^{n+1}(c)$ (\cite{An}). We extend all extrinsic geometric quantities such as the shape operator $A$, the principal directions $e_i$ and the principal curvatures $k_i$ to the 1-parameter family of immersions $(\phi_t)$. Using (\ref{e:inducmet}), the induced metric $\Phi_t^{\ast}\mathcal{G}$ is given by 
\[
\Phi^{\ast}_t\mathcal{G}=\mbox{diag}\Big(2k_1,\ldots,2k_n\Big).
\]
For every sufficiently small $t>0$, the volume of every Gauss map $\Phi_t$, with respect to the metric $G$, is
\begin{equation}\label{e:impeqt}
\mathop{\rm Vol}(\Phi_t)=\int_{\Sigma}\sqrt{|\mathop{\rm det}\Phi_t^{\ast}\mathcal{G}|}dV=2^{n/2}{\mathcal F}(\phi_t).
\end{equation}
If $\phi$ is a critical point of the functional ${\mathcal F}$, we have
\[
\partial_t(\mathop{\rm Vol}(\Phi_t))|_{t=0}=0,
\]
for any Hamiltonian variation of $\Phi$. Therefore, $\Phi$ is a Hamiltonian minimal submanifold with respect to the para-K\"ahler structure $({\mathcal G},{\mathcal J})$. The converse follows directly from (\ref{e:impeqt}).

\end{proof}

A consequence of Theorem \ref{t:hminimal} is the following corollary:
\begin{coro}
Suppose that a non-flat submanifold $\Sigma$ in ${\mathbb S}^{n+1}(c)$ satisfies the condition 
\begin{equation}\label{e:nicecondition}
\sum_{i=1}^n e_i\left(\frac{e_i(\sqrt{|K|})}{k_i}\right)=0,
\end{equation}
where $e_i$ are the principal directions with corresponding principal curvatures $k_i$ and $K$ is the Gaussian curvature. Then $\Sigma$ is a critical point of the functional ${\mathcal F}$.
\end{coro}
\begin{proof}
Consider the Gauss map $\Phi$ of a hypersurface $\Sigma$ in ${\mathbb S}^{n+1}(c)$ immersed by $\phi$ satisfying (\ref{e:nicecondition}). Using the Theorem \ref{t:hminimal}, we only need to show that $\Phi$ is Hamiltonian minimal. If ${\mathbb H}$ denotes the mean curvature vector field of $\Phi$ then (\ref{e:meancurv}) yields
\[
n\,\mbox{div}({\mathcal J}{\mathbb H})=\Delta\log |K|^{-1/2},
\]
where div and $\Delta$ are respectively the divergence operator and the Laplacian of $\Phi^{\ast}{\mathcal G}$ (which is given by (\ref{e:inducmet})). It is not hard for one to confirm that 
\[
\Delta\log |K|^{-1/2}=-\frac{1}{2\sqrt{|K|}}\sum_{i=1}^ne_i\left(\frac{e_i(\sqrt{|K|})}{k_i}\right),
\]
and the corollary follows.
\end{proof}

\end{document}